\documentclass{article}

\usepackage{amsmath, amssymb}
\usepackage{graphicx, wrapfig} 
\usepackage{amsthm}
\usepackage{url}

\theoremstyle{plain}
\newtheorem{thm}{Theorem}
\newtheorem{lem}{Lemma}

\theoremstyle{definition}

\newtheorem{prop}{Proposition}
\newtheorem{cor}{Corollary}

\title{Isolated regions of a link projection}
\author{Tumpa Mahato\thanks{Department of Mathematics, Indian Institute of Science Education and Research, Pune, India. Email: tumpa.mahato@students.iiserpune.ac.in} and Ayaka Shimizu\thanks{Osaka Central Advanced Mathematical Institute (OCAMI), 3-3-138, Sugimoto, Sumiyoshi-ku, Osaka, 558-8585, Japan. Email: shimizu1984@gmail.com}}
\date{\today}
\begin{document}
\maketitle

\begin{abstract}
A set of regions of a link projection is said to be isolated if any pair of regions in the set share no crossings. 
The isolate-region number of a link projection is the maximum value of the cardinality for isolated sets of regions of the link projection. 
In this paper, all the link projections of isolate-region number one are determined. 
Also, estimations for welded unknotting number and combinatorial way to find the isolate-region number are discussed, and a formula of the generating function of isolated-region sets is given for the standard projections of $(2,n)$-torus links.
\end{abstract}

\section{Introduction}

A {\it knot} is an embedding of a circle in $S^3$ and a {\it link} is an embedding of some circles in $S^3$. 
We assume a knot to be a link with one component. 
A {\it link projection} is a regular projection of a link on $S^2$ such that any intersection is a double point where two arcs intersect transversely. 
We call such an intersection a {\it crossing}. 
We call each connected part of $S^2$ divided by a link projection a {\it region}. 
It is well known that the number of regions is greater than the number of crossings by two\footnote{See, for example, \cite{ahara-suzuki}.} for each connected link projection.

Two regions $x_a$ and $x_b$ of a link projection $L$ are said to be {\it connected} when they share a crossing on their boundaries. 
Otherwise, they are said to be {\it disconnected}. 
For example, the regions $x_1$ and $x_2$ are connected, whereas $x_1$ and $x_8$ are disconnected in Figure \ref{fig-7crossings}. 
A set of regions $W= \{ x_a, x_b, \dots , x_c \}$ is said to be {\it isolated}, or an {\it isolated-region set}, if any pair of regions $x_{\alpha}$ and $x_{\beta} \in W$ are disconnected. 
We assume that the empty set $\emptyset$ is isolated. 
The {\it isolate-region number}, $I(L)$, of a link projection $L$ is the maximum value of the cardinality for all isolated-region sets $W$ of $L$. 
For example, the isolate-region number of the knot projection $K$ in Figure \ref{fig-7crossings} is three. 
\begin{figure}[h]
\centering
\includegraphics[width=25mm]{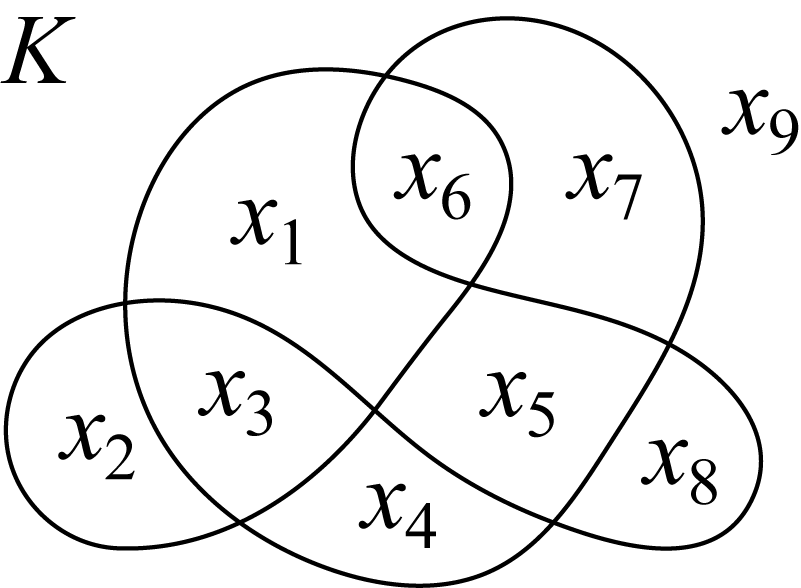}
\caption{$\{ x_3 , x_6 , x_8 \}$, $\{ x_1 , x_8 \}$ are isolated-region sets. }
\label{fig-7crossings}
\end{figure}

\noindent In \cite{warping-1}, a similar notion, the {\it independent region number} of a link projection $L$, denoted by $I\hspace{-1pt}R(L)$, was previously defined as follows: 
For a nontrivial link projection $L$, take a crossing $c$. 
Choose a set of independent regions of $L$ such that any of the regions does not have $c$ on the boundary. 
The independent region number $I\hspace{-1pt}R(L)$ is the maximum number of such regions for all crossings $c$. 
For example, the knot ptojection $K$ in Figure \ref{fig-7crossings} has independent region number three which is realized by taking the crossing shared by $x_1$, $x_3$, $x_4$ and $x_5$, and taking the regions $x_2$, $x_6$ and $x_8$. 
The independent region number was introduced to give lower and upper bounds for the ``warping degree'', $d(D)$, of a knot diagram $D$, which represents a complexity of a knot diagram  (\cite{lecture, warping-3}) as explained in Section \ref{section-wd} in this paper. 
The concept of the isolate-region number was used in the same paper without being named to give a lower bound of the independent region number and show the inequality
$$I(D)-1 \leq I\hspace{-1pt}R(D) \leq d(D) \leq c(D)-I\hspace{-1pt}R(D)-1 \leq c(D)-I(D)$$
for a non-trivial alternating knot diagram $D$, where $I(D)$, $I\hspace{-1pt}R(D)$ denote the isolate-region number, independent region number, respectively of the knot projection associated to $D$, and $c(D)$ denotes the number of crossings of $D$. 
In the above inequality, we can see that the independent region number $I\hspace{-1pt}R(D)$ gives better estimation for the warping degree than the isolate-region number $I(D)$. 
However, the definition of isolate-region number is simpler than the independent region number and suitable for calculation and investigation. 
Indeed, some properties of independent region number were obtained by finding natural properties of the isolate-region number in \cite{warping-1}. 
In this paper, we focus on the isolate-region number and prove the following theorem. 

\begin{thm}
A link projection $L$ has isolate-region number one if and only if $L$ is one of the link projections shown in Figure \ref{fig-IR1}.
\label{thm-IR1}
\end{thm}

\begin{figure}[h]
\centering
\includegraphics[width=50mm]{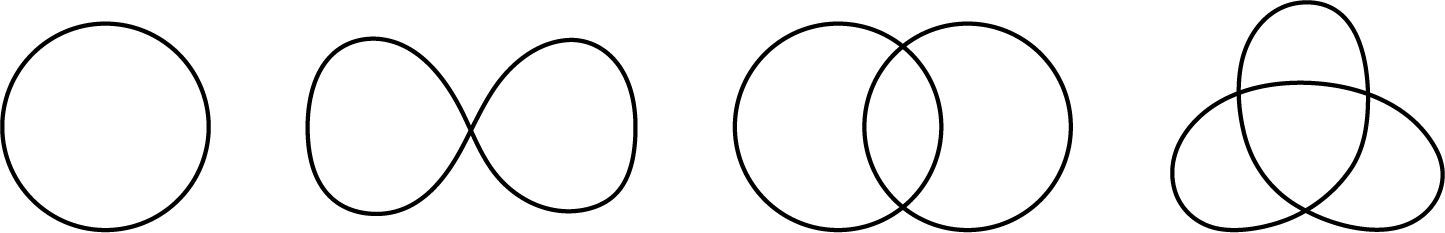}
\caption{All the link projections of isolate-region number one. }
\label{fig-IR1}
\end{figure}

\noindent The rest of the paper is organized as follows. 
In Section \ref{section-pf1}, the proof of Theorem \ref{thm-IR1} is given. 
In Section \ref{section-wd}, some properties of warping degree are discussed and an upper bound for the welded unknotting number for some special knot diagrams is given as the difference between the crossing number and the isolate-region number of the diagram. 
In Section \ref{section-graph}, a combinatorial method to find the isolate-region number using graphs is given. 
In Section \ref{section-function}, the generating function of isolated-region sets are discussed to see the distribution of the sets. 
In Section \ref{section-pf2}, a formula and recursive relationship of the generating functions for the standard projections of $(2,n)$-torus links are discussed.

\section{Proof of Theorem \ref{thm-IR1}} 
\label{section-pf1}

In this section, we prove Theorem \ref{thm-IR1}. 
For link projections $M$ and $N$, a {\it split sum} and a {\it connected sum} are link projections obtained from $M$ and $N$ as shown in Figure \ref{fig-composite}. 
We say a link projection $L$ is {\it connected} when $L$ is not a split sum of link projections. 
We say a link projection $L$ is {\it unconnected} when $L$ is not connected, namely, when $L$ is a split sum. 
Recall that we say two regions are {\it disconnected} when they share no crossings. 
We say a link projection $L$ is {\it composite}  when $L$ is a connected sum of two nontrivial link projections. 
We have the following lemma. 

\begin{figure}[h]
\centering
\includegraphics[width=85mm]{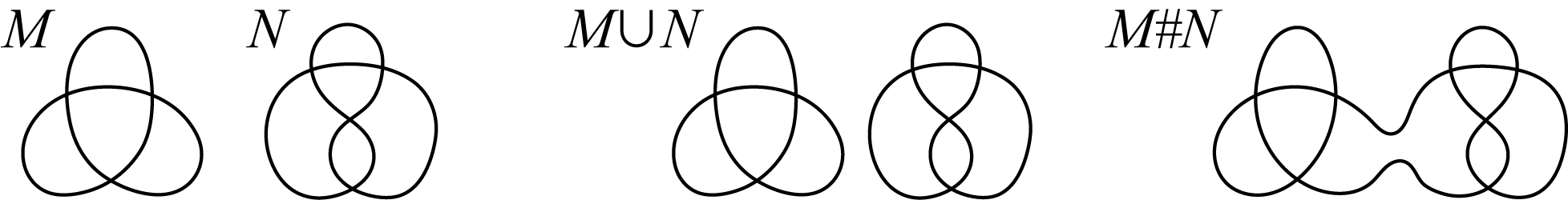}
\caption{A split sum $M \cup N$ and a connected sum $M \# N$ of knot projections $M$ and $N$. }
\label{fig-composite}
\end{figure}

\begin{lem}
Every unconnected or composite link projection $L$ has $I(L) \geq 2$. 
\label{lem-conn-sum}
\end{lem}

\begin{proof}
When $L$ is a split sum of $M$ and $N$, take a region from $M$ and a region from $N$, neither of which are the region bounded by both $M$ and $N$. 
When $L$ is a connected sum of $M$ and $N$, take a region from $M$ and a region from $N$, neither of which are the shared region in the connected sum $M \# N$. 
Then, we have the inequality $I(L) \geq 2$. 
\end{proof}

\noindent A crossing of a link projection is said to be {\it reducible} if it has exactly three regions around it. 
A link projection $L$ is {\it reducible} when $L$ has a reducible crossing. 
We say $L$ is {\it irreducible} when $L$ is not reducible. 
Since a reducible link projection can be assumed to be a connected sum when it has at least two crossings, we have the following corollary. 

\begin{cor}
Every reducible link projection $L$ with two or more crossings has $I(L) \geq 2$. 
\label{cor-reducible}
\end{cor}

\noindent Next, we divide a link projection into two tangles by giving a circle $C$, where $C$ intersects edges transversely. 
We assume each tangle $T$ is on a disk $D$ with $\partial D =C$, and call each connected region of $D$ divided by $T$ a {\it region} of $T$. 
In particular, if a region touches $\partial D$, we call it an {\it open region}, and otherwise we call it a {\it closed region}. 
We remark that any pair of closed regions in opposite sides of $C$ are disconnected. 
We show the following.

\begin{lem}
For any connected irreducible link projection with four or more crossings, we can draw a circle $C$ so that the tangle on each side of $C$ has a closed bigon or trigon. 
\label{lem-circ4}
\end{lem}

\begin{proof}
We use the following formula of Adams, Shinjo and Tanaka from \cite{AST}, 
\begin{align*}
2C_2 +C_3 = 8 + C_5 + 2C_6 + 3C_7 + \dots
\end{align*}
for the number $C_k$ of $k$-gons of a connected irreducible link projection. 
From the formula, we have $2C_2 + C_3 \geq 8$. 
We will proceed by analyzing cases according to the number of bigons in the connected irreducible link projection with four or more crossings. \\

\noindent Case 1: Suppose $C_2 =0$. 
Using the inequality above, our supposition implies that $C_3 \geq 8$. 
Draw a circle $C$ around a trigon as shown in Figure \ref{fig-23}, (I\hspace{-0.2mm}I). 
Then the tangle in the figure has a closed trigon. 
Since $C_3 \geq 8 >7$ and  there are at most six regions around the trigon, the tangle on the other side has a closed trigon, too. \\
\begin{figure}[h]
\centering
\includegraphics[width=50mm]{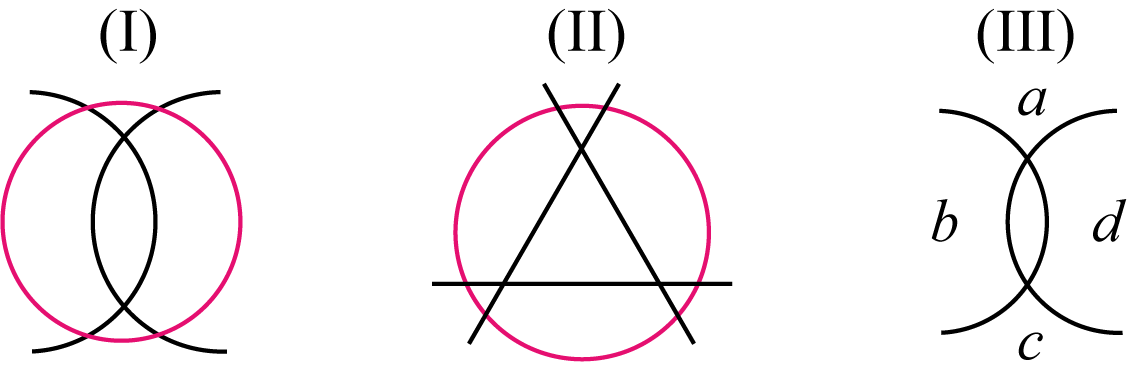}
\caption{(I) A circle around a bigon. (I\hspace{-0.2mm}I) A circle around a trigon. (I\hspace{-0.2mm}I\hspace{-0.2mm}I) Labeled regions $a, b, c, d$ around a bigon. }
\label{fig-23}
\end{figure}

\noindent Case 2: Suppose $C_2 =1$. 
Using the inequality above, our supposition implies that $C_3 \geq 6$. 
Draw a circle $C$ around the bigon as shown in Figure \ref{fig-23}, (I). 
Then, the tangle on the other side has closed trigons because $C_3 \geq 6 >4$ and there are at most four regions around a bigon. \\
\noindent Case 3: Suppose $C_2 =2$. 
Using the inequality above, our supposition implies that $C_3 \geq 4$. 
Draw $C$ as shown in Figure \ref{fig-23}, (I). 
Since $C_2 + C_3 \geq 6 >5$, the tangle on the other side has a closed bigon or trigon, too. \\
\noindent Case 4: Suppose $C_2 =3$. 
Using the inequality above, our supposition implies that $C_3 \geq 2$. 
We consider two subcases, first we consider $C_3 =2$ and second $C_3 \geq 3$. 
First suppose $C_3 =2$. 
Choose any bigon in the projection and see (I\hspace{-0.2mm}I\hspace{-0.2mm}I) in Figure \ref{fig-23}. 
The regions labeled $a$ and $c$ are either the same regions in the link projection or they are distinct regions. 
If $a$ and $c$ are the same region in the projection, we can draw a circle $C$ as in Figure \ref{fig-23}, (I) since we have $C_2 + C_3 =5 >4$ and the bigon has only three regions around it. 
If $a$ and $c$ are distinct regions in the projection, then neither $b$ nor $d$ is a bigon; 
if $b$ or $d$ is a bigon, then $a$ and $c$ are the same region in the projection. 
Suppose $a$ and $c$ are distinct, connected and both bigons. 
Then the projection will be a knot projection with only three crossings and be inconsistent with the condition of four or more crossings. 
Hence this does not happen. 
When $a$ and $c$ are distinct and connected, and either $a$ or $c$ is not a bigon, take $C$ as shown in Figure \ref{fig-23}, (I). 
Then the other side has a closed bigon, too. 
Next, suppose $a$ and $c$ are distinct regions and they are disconnected bigons. 
Then both $b$ and $d$ have the four crossings which are the crossings of the two distinct bigons $a$ and $c$ on their boundary. 
Hence $b$ and $d$ are $n$-gons for $n \geq 4$. 
Take $C$ as in Figure \ref{fig-23}, (I). 
Then the other side of $C$ has closed trigons.
Suppose $a$ and $c$ are disconnected and either $a$ or $c$ is not a bigon. 
Draw $C$ as in Figure \ref{fig-23}, (I). 
Then the other side has a closed bigon. 
Second, suppose $C_2 =3$ and $C_3 \geq 3$. 
We can draw $C$ as in Figure \ref{fig-23}, (I), because $C_2 + C_3 \geq 6 >5$. \\
Case 5: Suppose $C_2=4$. 
Using the inequality above, our supposition implies that $C_3 \geq 0$. 
Choose any bigon and see (I\hspace{-0.2mm}I\hspace{-0.2mm}I) in Figure \ref{fig-23}. 
If the regions $a$ and $c$ are the same region in the projection and it is a bigon, then the projection will be a link projection with only two crossings. 
Hence this is not allowed to happen. 
If $a$ and $c$ are the same region in the projection and it is not a bigon, we can draw $C$ as in Figure \ref{fig-23}, (I) since inside $C$ in the figure has at most three bigons. 
If $a$ and $c$ are distinct regions in the projection, we can draw $C$ as shown in Figure \ref{fig-23}, (I) since neither $b$ nor $d$ is a bigon, and inside $C$ in the figure has at most three bigons, too. \\
Case 6: Suppose $C_2 \geq 5$. 
Choose any bigon and see (I\hspace{-0.2mm}I\hspace{-0.2mm}I) in Figure \ref{fig-23}. 
If both $b$ and $d$ are bigons, then the projection will be a link projection with only two crossings. Hence at least one of $b$ and $d$ is an $n$-gon for $n \geq 3$. 
We can draw $C$ as in Figure \ref{fig-23}, (I) since inside $C$ in the figure has at most four bigons. 
\end{proof}

\noindent By Lemmas \ref{lem-conn-sum}, \ref{lem-circ4} and Corollary \ref{cor-reducible}, we have the following.

\begin{cor}
Any link projection $L$ with $4$ or more crossings has $I(L) \geq 2$. 
\label{cor-4-2}
\end{cor}

\noindent Now we prove Theorem \ref{thm-IR1}. \\

\noindent {\it Proof of Theorem \ref{thm-IR1}.}
Since every unconnected projection has isolate-region number two or more by Lemma \ref{lem-conn-sum}, it is sufficient to discuss connected projections. 
When the crossing number $n=0, 1$, there are just two link projections, the first and second knot projections in Figure \ref{fig-IR1}, and they have isolate-region number one. 
When $n=2,3$, we have just two reduced link projections, the third and fourth projections in Figure \ref{fig-IR1}, and they have isolate-region number one. 
By Corollaries \ref{cor-reducible} and \ref{cor-4-2}, there are no link projections of isolate-region number one for link projections with four or more crossings or reducible projections with two or more crossings. 
\hfill$\square$ \\

\noindent A {\it checkerboard coloring} is a coloring to a link projection by two colors such that each pair of regions sharing an edge have different colors. 
It is well known that every link projection admits a checkerboard coloring. 
In addition to Corollary \ref{cor-4-2}, we have the following proposition regarding the location of isolated regions. 

\begin{prop}
If an irreducible link projection has an $n$-gon for $n \geq 4$, there is a pair of disconnected regions around the $n$-gon. 
\end{prop}

\begin{proof}
Take points $a, b, c, d$ on edges of the $n$-gon and call the regions $A, B, C, D$, as shown in Figure \ref{fig-ngon}. 
Since the knot projection is irreducible, the regions $A$ and $B$ are distinct regions. 
Similarly, $B$ and $C$, and $C$ and $D$ are distinct regions. 
The regions $A$ and $D$ are also distinct since $n \geq 4$. 
We will consider two cases, first we consider the case when the regions $A$ and $C$ are the same region in the projection and second they are distinct regions. \\
Case 1: Suppose $A$ and $C$ are the same region in the projection. 
We can draw a simple curve $\alpha$ connecting the points $a$ and $c$ inside the region of $A$ and $C$. 
Since $B$ and $D$ cannot touch $\alpha$, they are disconnected. \\
Case 2: Suppose $A$ and $C$ are distinct regions. 
We consider two subcases, first we consider that $A$ and $C$ are disconnected and second they are connected. 
First suppose $A$ and $C$ are disconnected. 
Then, they are exactly the disconnected regions around the $n$-gon. 
Second, we consider $A$ and $C$ are connected. 
We can draw a simple curve connecting $a$ and $c$ passing through $A$, $C$ and one crossing $p$ shared by $A$ and $C$. 
Then, the regions $B$ and $D$ are distinct regions because they are in the opposite sides of the curve $\alpha$ and $\alpha$ passes through only the two regions $A$ and $C$. 
Moreover, $B$ and $D$ are disconnected. 
Otherwise, $B$ and $D$ must share the crossing $p$ to be connected. 
This contradicts that the link projection admits a checkerboard coloring; 
the regions $A$, $B$, $C$ and $D$ should be given the same color in a checkerboard coloring as they are neighboring the same $n$-gon. 
On the other hand, they are all gathering around the crossing $p$, and this implies the colors of $A$, $C$ and $B$, $D$ should differ. 
\begin{figure}[h]
\centering
\includegraphics[width=25mm]{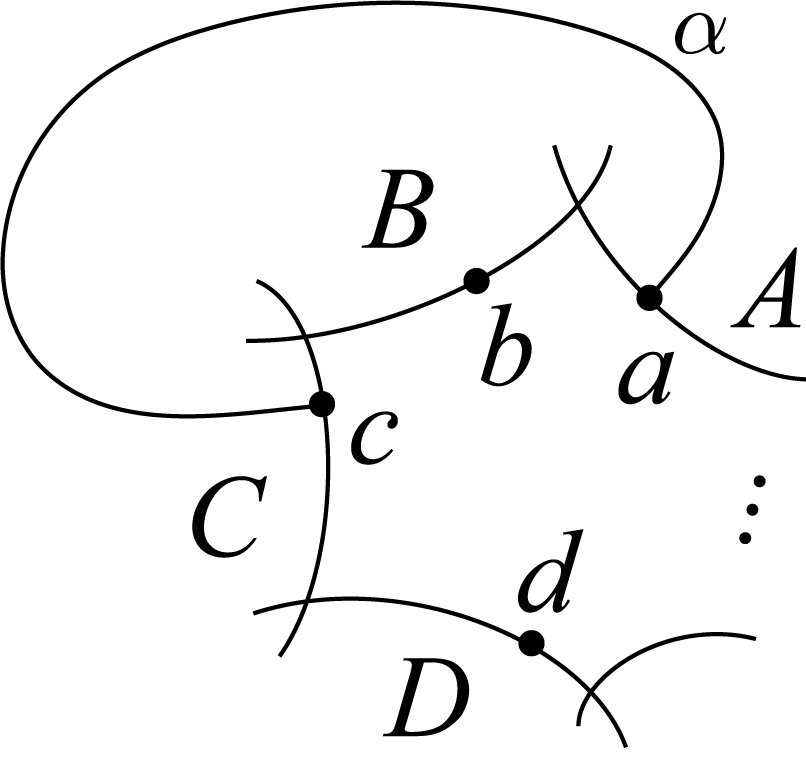}
\caption{Regions $A, B, C, D$ around an $n$-gon.}
\label{fig-ngon}
\end{figure}
\end{proof}

\noindent Also, we have the following corollary from Lemma \ref{lem-circ4}. 

\begin{cor}
Any connected irreducible link projection $L$ with $13$ or more crossings has $I(L) \geq 3$. 
\end{cor}

\begin{proof}
By Lemma \ref{lem-circ4}, $L$ has a pair of disconnected regions, each of which is either a bigon or a trigon, say $r_1$ and $r_2$. 
Then $L$ has at most $12$ regions which are connected to either $r_1$ or $r_2$. 
Since $L$ has $15$ or more regions, $L$ has at least one region, say $r_3$, which is neither $r_1$ nor $r_2$ and is connected to neither $r_1$ nor $r_2$.
\end{proof}

\section{Warping degree, welded unknotting number and isolate-region number of a knot diagram}
\label{section-wd}

In this section, we discuss the warping degree and welded unknotting number of a knot diagram and their relations to the isolate-region number. 
In Subsection \ref{subsection-wd}, we discuss the relation between the warping degrees of a knot diagram and knot projection. 
In Subsection \ref{subsection-welded}, we discuss the relation between the welded unknotting number, warping degree and isolate-region number of a knot diagram.

\subsection{Warping degree}
\label{subsection-wd}

Let $D$ be an oriented knot diagram with a base point $b$ on an edge. 
A crossing point of $D$ is said to be a {\it warping crossing point of $D_b$} if we encounter the crossing as an under crossing first when we travel $D$ from $b$ with the orientation. 
The {\it warping degree of $D_b$}, $d(D_b)$, is the number of the warping crossing points of $D_b$. 
We can represent all the warping degrees by the ``warping degree labeling'' as shown in Figure \ref{fig-wdl}.

\begin{figure}[h]
\centering
\includegraphics[width=75mm]{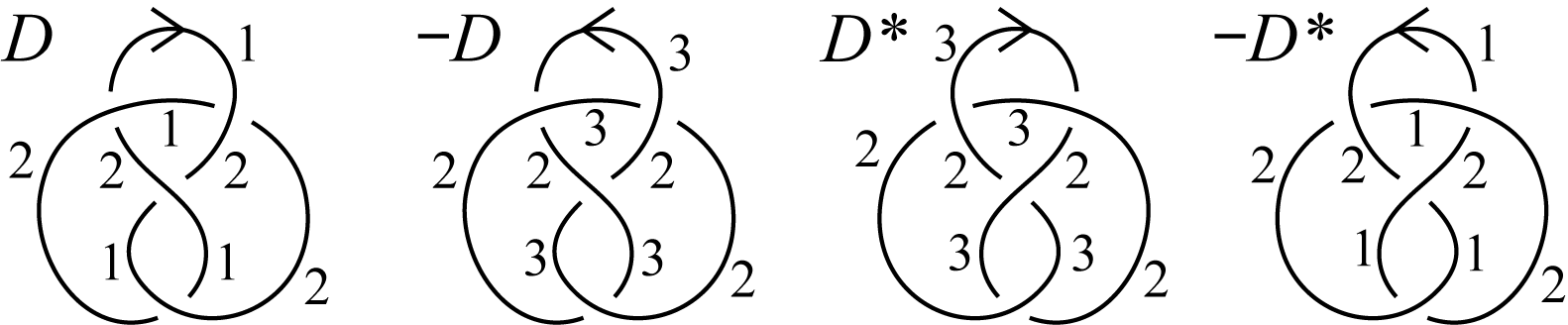}
\caption{Warping degree labeling. Each label on an edge indicates the warping degree where a base point is on the edge.}
\label{fig-wdl}
\end{figure}

\noindent The {\it warping degree of $D$}, $d(D)$, is the minimal value of $d(D_b)$ for all base points $b$ of $D$ (\cite{lecture}). 
As shown in Figure \ref{fig-wdl}, warping degree depends on the orientation; 
we have $d(D)=1$, whereas $d(-D)=2$ in the figure, where $-D$ is $D$ with orientation reversed. 
In this paper, we define the {\it unoriented warping degree}, denoted by $\overline{d}(D)$, to be $\overline{d}(D)= \min \{ d(D), d(-D) \}$. 
For a knot diagram $D$, we denote by $|D|$ the knot projection which is obtained from $D$ by forgetting the crossing information. 
The {\it warping degree of a projection $P$}, $d(P)$, is the minimal value of $d(D)$ for all oriented alternating diagrams $D$ with $|D|=P$. 
Although each projection has four oriented alternating diagrams (see Figure \ref{fig-wdl}), the following proposition claims it is sufficient to check only one alternating diagram with two orientations.

\begin{prop}
The equality $\overline{d}(D)=d(P)$ holds when $D$ is an alternating diagram with $|D|=P$. 
\end{prop}

\begin{proof}
As shown in \cite{warping-3}, we have $d(-D_b)=c(D)-d(D_b)$ for any oriented knot diagram $D$ with a base point $b$ because the order of over crossing and under crossing is switched at every crossing. 
We also have $d({D_b}^*)= c(D)-d(D_b)$ for the mirror image $D^*$ because the crossing information is changed at each crossing. 
Therefore, we have $d({D_b}^*)=d(-D_b)$ and $d(D^*)=d(-D)$. 
Hence $d(P) = \min \{ d(D), \ d(-D), \ d(D^*), \ d(-D^*)  \} = \min \{ d(D), \ d(-D) \} = \overline{d}(D)$ for any alternating diagram $D$ with $|D|=P$. 
\end{proof}

\noindent A nonalternating knot diagram $D$ is said to be {\it almost alternating} if $D$ becomes alternating by a single crossing change. 
For almost alternating diagrams, we have the following proposition. 

\begin{prop}
Let $D$ be an almost alternating diagram with $|D|=P$. 
Then $\overline{d}(D) \leq d(P)$ holds. 
In particular, if an almost alternating diagram $D$ has no monogons, $\overline{d}(D)=d(P)-1$.
\label{lem-almalt}
\end{prop}

\noindent To prove Proposition \ref{lem-almalt}, we prepare a further formula. 
Since we have $d(-D)= \min_b d(-D_b) = c(D) - \max_b d(D_b)$ for any diagram $D$, we have
\begin{align*}
\overline{d}(D)= \min \{ d(D), \ d(-D) \} = \min \{ \min_b d(D_b), \ c(D)- \max_b d(D_b) \} 
\end{align*}
for any oriented knot diagram $D$. 
Hence, we can find the value of $d(P)$ only from one oriented alternating diagram $D$ with $|D|=P$ by looking at the minimum and maximum values of warping degree labeling. 
We prove Proposition \ref{lem-almalt}. \\

\noindent {\it Proof of Proposition \ref{lem-almalt}.} 
Let $D$ be an almost alternating knot diagram which is obtained from an oriented alternating diagram $D^A$ by a single crossing change at a crossing $c$. 
Let $C_1$ (resp. $C_2$) be a subdiagram of $D^A$ from a point just after the over crossing (resp. the under crossing) of $c$ to a point just before the under crossings (resp. the over crossing) of $c$, as shown in Figure \ref{fig-da-d}. 
\begin{figure}[h]
\centering
\includegraphics[width=45mm]{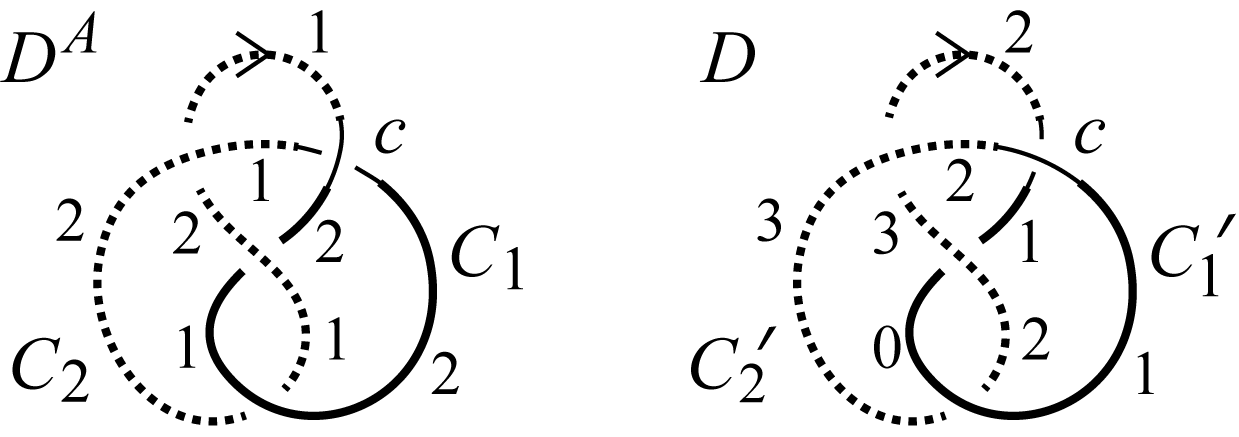}
\caption{An alternating diagram $D^A$ and an almost alternating diagram $D$ which are related by a crossing change at the crossing $c$. The subdiagrams $C_1$ and ${C_1}'$ are thickened and $C_2$ and ${C_2}'$ are dotted.}
\label{fig-da-d}
\end{figure}
Let ${C_1}'$ (resp. ${C_2}'$) be the corresponding subdiagrams of $D$ to $C_1$ (resp. $C_2$). 
Each warping degree labeling on ${C_1}'$ is smaller than the corresponding labeling on $C_1$ by one because $c$ becomes a non-warping crossing point for base points on ${C_1}'$ by the crossing change at $c$ as mentioned in \cite{quan, warping-p}. 
In the same way, each labeling on ${C_2}'$ is greater than $C_2$ by one.
Let $d(D^A)=i$. 
Then we have $d(-D^A)=c(D)-i-1$, and $\overline{d}(D^A)= \min \{ i, \ c(D)-i-1 \}$. 
We will proceed by analyzing cases according to whether $C_1$ and $C_2$ have crossings. 
In each case, the relation between the unoriented warping degrees of $D$ and $D^A$ will be discussed by considering the warping degree labeling to determine the relation between $\overline{d}(D)$ and $d(P)$ ($=\overline{d}(D^A)$), where $P$ is the knot projection with $|D|=|D^A|=P$. \\
Case 1: Suppose both $C_1$ and $C_2$ have crossings. 
Then their warping degree labeling are consisting of $i$ and $i+1$. 
After the crossing change at the crossing $c$ of $D^A$, the labels on ${C_1}'$ are $i-1, \ i$ and the labels on ${C_2}'$ are $i+1, \ i+2$. 
Hence $d(D)=i-1 =d(D^A)-1$ and $d(-D)= c(D)-i-2 = d(-D^A)-1$, and therefore $\overline{d}(D)= \overline{d}(D^A)-1$. \\
Case 2: Suppose $C_1$ has no crossings and $C_2$ has crossings. 
The label of $C_1$ is only $i+1$ and the labels of $C_2$ are $i+1, \ i+2$, and the labels of ${C_1}'$ is $i$ and that of ${C_2}'$ are $i+1, \ i+2$. 
Hence $d(D)=i, \ d(-D)=c(D)-i-2$, and $\overline{d}(D) \leq \overline{d}(D^A)$. \\
Case 3: Suppose $C_1$ has crossings and $C_2$ has no crossings. 
The labels of $C_1$ are $i-1, \ i$ and that of $C_2$ is only $i$, and the labeling of ${C_1}'$ are $i-1, \ i$ and that of ${C_2}'$ is $i+1$. 
Hence $d(D)=i-1, \ d(-D)= c(D) -i-1$, and $\overline{d}(D) \leq \overline{d}(D^A)$. \\
Case 4: Suppose neither $C_1$ nor $C_2$ has crossings. 
Then $c(D)=1$, and $D$ is not an almost alternating diagram. \\
Therefore, we have $\overline{d}(D) \leq \overline{d}(D^A)=d(P)$, where $P=|D|$. 
In particular, if $D$ has no monogons, then both $C_1$ and $C_2$ must have crossings for any crossing $c$, and this case applies Case 1. 
\hfill$\square$ \\

\noindent Next, we show the following proposition. 

\begin{prop}
Let $P$ be an irreducible knot projection with $c(P) \leq  9$. 
The inequality $\overline{d}(D) \leq d(P)$ holds for any knot diagram $D$ with $|D|=P$ if and only if $P$ is not one of the four projections illustrated in Figure \ref{fig-4proj}.
\label{lem-wd9}
\end{prop}
\vspace{3mm}

\begin{figure}[h]
\centering
\includegraphics[width=65mm]{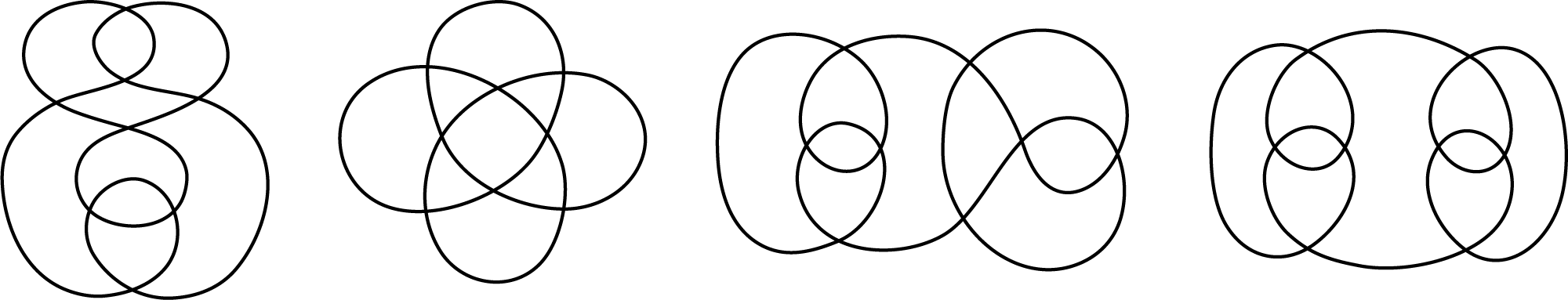}
\caption{The four exceptional projections.}
\label{fig-4proj}
\end{figure}

\noindent To prove Proposition \ref{lem-wd9}, we prepare Lemmas \ref{lem-d2}, \ref{lem-dc}, \ref{lem-spn}. 
All the irreducible knot projections with warping degree 1 and 2 are determined in \cite{warping-2}. 

\begin{lem}{\cite{warping-2}}
Let $P$ be an irreducible knot projection. \\
(1) $d(P)=1$ if and only if $P$ is one of the two projections depicted in Figure \ref{fig-d1}. \\
(2) $d(P)=2$ if and only if $P$ is one of the 16 projections depicted in Figure \ref{fig-d2}.
\label{lem-d2}
\end{lem}

\begin{figure}[h]
\centering
\includegraphics[width=20mm]{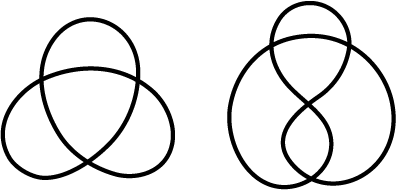}
\caption{All the irreducible knot projections of warping degree one.}
\label{fig-d1}
\end{figure}

\begin{figure}[h]
\centering
\includegraphics[width=100mm]{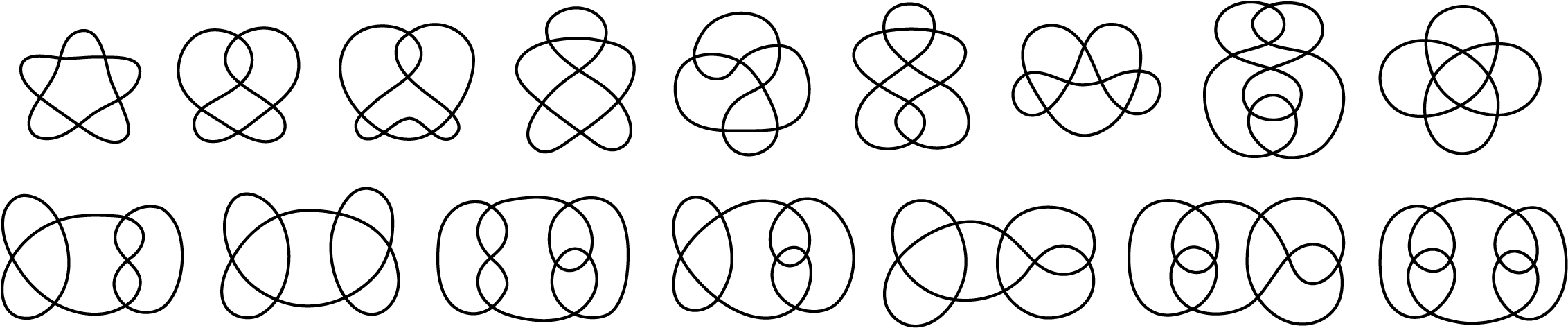}
\caption{All the irreducible knot projections of warping degree two.}
\label{fig-d2}
\end{figure}

\noindent The following lemma gives an upper bounds for $\overline{d}(D)$ and $d(P)$. 

\begin{lem}
The inequality $\overline{d}(D) \leq (c(D)-1)/2$ holds for any knot diagram $D$ with $c(D) \geq 1$. 
The inequality $d(P) \leq (c(P)-1)/2$ holds for any knot projection $P$ with $c(P) \geq 1$.
\label{lem-dc}
\end{lem}

\begin{proof}
By the inequality $d(D)+d(-D) \leq c(D)-1$ shown in \cite{warping-3}, we have $2 \overline{d}(D) \leq c(D)-1$, and therefore $\overline{d}(D) \leq (c(D)-1)/2$. 
Since this also holds for alternating diagrams $D$, we have $d(P) \leq (c(P)-1)/2$. 
\end{proof}

\noindent The {\it span}, $s(D)$, of an oriented knot diagram $D$ is defined to be the difference $s(D)= \max_b d(D_b) - \min_b d(D_b)$.
By definition, we have $s(D)=s(-D)$. 
Also, we have the following. 

\begin{lem}{\cite{warping-p}}
The span $s(D)$ is one if and only if $D$ is an alternating diagram.
\label{lem-spn}
\end{lem}

\noindent Using the above lemmas, we show Proposition \ref{lem-wd9}. \\

\noindent {\it Proof of Proposition \ref{lem-wd9}}. 
We will show that there exist no non-alternating diagrams $D$ with $|D|=P$ such that $\overline{d}(D) > d(P)$ under the condition. 
We will proceed by analyzing cases according to the number of crossings of $P$. \\
Case 1: Suppose $c(P)=3$ or $4$. We have $d(P)=1$ by Lemma \ref{lem-dc}, and there are no diagrams $D$ with $|D|=P$ satisfying $\overline{d}(D)>1$ by Lemma \ref{lem-dc}. \\
Case 2: Suppose $c(P)=5$ or $6$. We have $d(P)=2$ by Lemmas \ref{lem-d2}, \ref{lem-dc}, and there are no diagrams $D$ satisfying $\overline{d}(D) >2$ by Lemma \ref{lem-dc}. \\
Case 3: Suppose $c(P)=7$. We have $d(P)=2$ or $3$ by Lemmas \ref{lem-d2}, \ref{lem-dc}. 
There are no diagrams $D$ with $|D|=P$ satisfying $\overline{d}(D) >3$ by Lemma \ref{lem-dc}. 
When $d(P)=2$, each oriented alternating diagram $D^A$ with $|D^A|=P$ has warping degree labeling consisting of the numbers $2, 3$ or $4,5$. 
Here, recall that each alternating knot diagram has warping degree labeling consisting of two labels by Lemma \ref{lem-spn}. 
There are no nonalternating diagrams $D$ with $|D|=P$ such that the minimum label is greater than $2$ and the maximum label is smaller than $5$, because $s(D)>1$ by Lemma \ref{lem-spn}. 
Hence, we have $\overline{d}(D) \leq d(P)$ for any $D$ with $|D|=P$ in this case. \\
Case 4: Suppose $c(P)=8$. We have $d(P)=2$ or $3$ by Lemmas \ref{lem-d2}, \ref{lem-dc}. 
When $d(P)=2$, $P$ is one of the four knot projections depicted in Figure \ref{fig-4proj} by Lemma \ref{lem-d2}. 
For each knot projection $P$, we have a diagram $D$ with $|D|=P$ such that $\overline{d}(D)=3 > 2=d(P)$ as shown in Figure \ref{fig-4diag}. 
\begin{figure}[h]
\centering
\includegraphics[width=80mm]{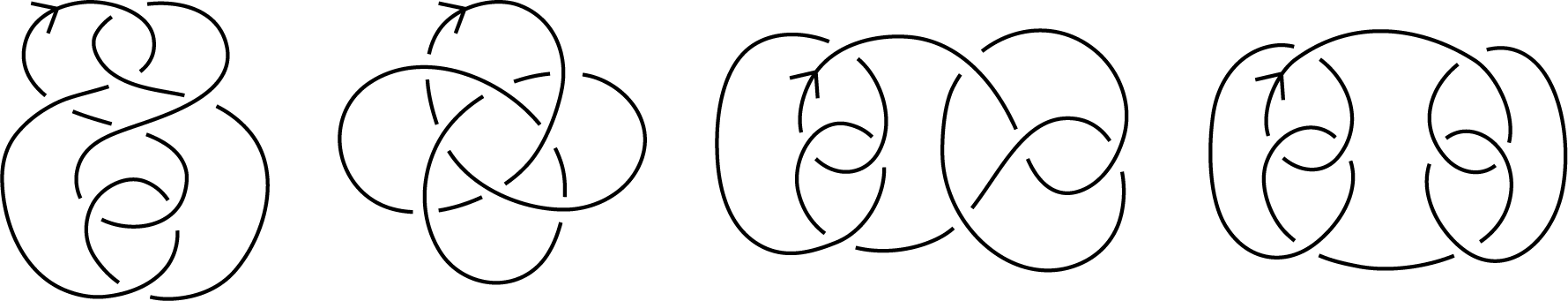}
\caption{Diagrams of warping degree three.}
\label{fig-4diag}
\end{figure}
When $d(P)=3$, the warping degree labeling on each oriented alternating diagram $D^A$ with $|D^A|=P$ is consisting of $3,4$ or $4,5$, and there are no nonalternating diagram $D$ such that the minimum label is greater than $3$ and the maximum label is less than $5$ since $s(D)>1$ by Lemma \ref{lem-spn}. \\
Case 5: Suppose $c(P)=9$. We have $d(P)=3$ or $4$ by Lemmas \ref{lem-d2}, \ref{lem-dc}. 
There are no diagrams $D$ such that $\overline{d}(D)>4$ by Lemma \ref{lem-dc}. 
When $d(P)=3$, each oriented alternating diagram $D^A$ has warping degree labeling consisting of $3,4$ or $5,6$. 
There are no nonalternating diagrams $D$ such that the minimum label is greater than $3$ and the maximum label is less than $6$ since $s(D)>1$ by Lemma \ref{lem-spn}. \hfill$\square$ \\

\subsection{Welded unknotting number}
\label{subsection-welded}

In this subsection, we apply results in Subsection \ref{subsection-wd} to a study of welded knot diagrams. 
A {\it welded knot diagram} is a diagram obtained from a classical knot diagram by replacing some classical crossings with ``welded crossings''. 
A {\it welded knot} (\cite{fenn, rourke}) is the equivalence class of welded knot diagrams related by classical and virtual Reidemeister moves and the over forbidden move. 
The {\it welded unknotting number}, $u_w (D)$, of a welded knot diagram $D$ is the minimum number of classical crossings which are needed to be replaced with a welded crossing to obtain a diagram of the trivial knot (\cite{non-tri}). 
Let $D$ be a classical knot diagram. 
It is known that $u_w (D) \leq \overline{d}(D)$ (\cite{unknotting-w, non-tri, s-virtual}), as a relation between the welded unknotting number and the warping degree. 
On the other hand, the inequality $I(P) \leq d(P) \leq c(P)-I(P)$, the relation between the warping degree and the isolate-region number, was shown in \cite{warping-3}. 
For a knot diagram $D$, let $I(D)=I(P)$, where $P=|D|$. 
From Propositions \ref{lem-almalt}, \ref{lem-wd9}, we have the following corollaries which are estimations of the welded unknotting number in terms of isolated regions. 

\begin{cor}
Let $D$ be an almost alternating diagram. 
We have $u_w (D) \leq c(D)-I(D)$. 
\label{cor-almalt}
\end{cor}

\begin{cor}
Let $D$ be an irreducible knot diagram of $c(D) \leq 9$ such that $|D|$ is not one of the projections in Figure \ref{fig-4proj}. 
We have $u_w (D) \leq c(D)-I(D)$. 
\label{cor-wd9}
\end{cor}

\section{Region-connect graph and its complement}
\label{section-graph}

In this section, region-connect graph and region-disconnect graph are defined and investigated. 
Then an algorithm to find the isolate-region number is given. 
For a graph $G$, we say a graph $G'$ is a {\it subgraph} of $G$ and denote by $G' \subseteq G$ if $G'=G$ or $G'$ is obtained from $G$ by deleting some vertices and edges.
For a link projection $L$ with $n$ regions, give labels $x_1, x_2, \dots , x_n$ to the regions. 
The {\it region-connect graph} $G$ of $L$ is the graph with the vertices labeled $x_1, x_2, \dots ,x_n$ and the edges $x_{\alpha} x_{\beta}$ for all the pairs of connected regions $x_{\alpha}$ and $x_{\beta}$ (see Figure \ref{fig-7crossings-gb}).
The {\it region-disconnect graph}, $\overline{G}$, of $L$ is the complement graph of $G$. 
By definition, $\overline{G}$ represents which pairs of regions are disconnected. 
Then we can find the isolated-region sets of a link projection $L$ by looking at $\overline{G}$ as follows. 
Find a subgraph of $\overline{G}$ which is a complete graph $K_r$. 
Since the corresponding $r$ regions to the vertices of $K_r$ are disconnected each other in $L$, the set of them is isolated. 
Moreover, the maximum size of a complete graph $K_r \subseteq \overline{G}$ is identical to the isolate-region number $I(K)$. 
For example, we obtain $I(K)=3$ for the knot projection $K$ in Figure \ref{fig-7crossings-gb} because the graph $\overline{G}$ includes $K_3$ and does not include $K_r$ for $r \geq 4$.

\begin{figure}[h]
\centering
\includegraphics[width=80mm]{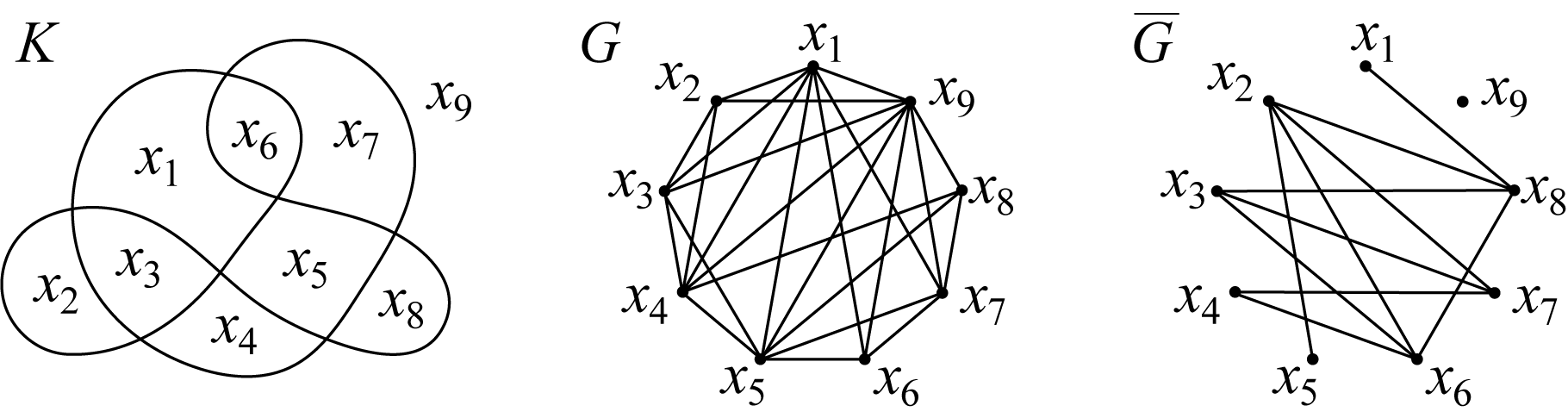}
\caption{A knot projection $K$ and its graphs $G$ and $\overline{G}$. }
\label{fig-7crossings-gb}
\end{figure}

\section{I-generating function}
\label{section-function}

In this section, we explore the generating functions for isolate-region sets of a link projection. 
Let $L$ be a link projection. 
Let $R= \{ x_1, x_2, \dots ,x_n \}$ be the set of all regions of $L$. 
Let $P(R)$ be the power set of $R$. 
Now suppose we want to know how isolated-region sets are distributed in $P(R)$. 
First, each set in $P(R)$ contains a set of regions from $R$. 
Hence, all the elements in $P(R)$ can be seen as the terms of the following polynomial. 

\begin{align*}
h &=(x_1^0 +x_1^1)(x_2^0 +x_2^1)(x_3^0 +x_3^1) \cdots (x_n^0 + x_n^1) \\
&=(1+x_1)(1+x_2)(1+x_3) \cdots (1+x_n) \\
&= 1 + x_1 + x_2 + \cdots +x_1 x_2 + x_1 x_3 + \cdots + x_1 x_2 x_3 + \cdots +x_1 x_2 \dots x_n ,
\end{align*}

\noindent where $x_i^1$ or $x_i^0$ denotes if the region $x_i$ is chosen or not respectively. 
For $h$, delete the terms which include the pair of connected regions, i.e., the pair of the letters which are connected in $G$. 
Then we obtain a polynomial $g$ which represents all the isolated-region sets. 
When we want to know just the number of isolated-region sets, replace each $x_i$ with $x$. 
Then, we obtain the generating function $f(L)$ for the isolated-region sets. 
In this paper, we call $f(L)$ the {\it I-generating function} of $L$. 
By definition, the maximum degree, $\text{maxdeg}f(L)$, indicates the isolate-region number $I(L)$ of $L$. 
We have the following proposition. 

\vspace{3mm}
\begin{prop}
For each I-generating function $f(L)=a_0 +a_1 x +a_2 x^2 +a_3 x^3 + \dots + a_r x^r$ of a connected link projection $L$ with $c$ crossings, $a_0 =1$ and $a_1 = c+2$. 
All the terms from degree 0 to $\text{maxdeg}f(L)=r$ are nonzero. 
\label{prop-func}
\end{prop}
\vspace{3mm}

\begin{proof}
We notice that each coefficient $a_j$ is the number of isolated-region sets consisting of $j$ regions. 
In terms of the region-disconnect graphs $\overline{G}$, $a_j$ is the number of complete graphs of $j$ vertices that are subgraph of $\overline{G}$. 
The first claim is obvious from the fact that we have one $\emptyset$ and the number of regions of $L$ is $c+2$. 
For the second claim, when $\overline{G}$ includes a complete graph $K_r$ ($r \geq 2$), then $\overline{G}$ also includes $K_{r-1}$ because $K_r \supseteq K_{r-1}$. 
\end{proof}
\vspace{3mm}

\noindent For connected irreducible link projections, we have an upper and lower bounds for $a_2$ as follows\footnote{We note that Proposition \ref{prop-a2} does not include Corollary \ref{cor-4-2} because the proposition claims just $a_2 \geq (c^2 -5c+2)/2=-2$ when $c=4$.}.

\vspace{3mm}
\begin{prop}
Let $L$ be a connected irreducible link projection with $c$ crossings and let $f(L)=a_0 + a_1 x + a_2 x^2 + \dots + a_r x^r$ be the I-generating function of $L$. We have 
\begin{align*}
\frac{c^2 -5c+2}{2} \leq a_2 \leq \frac{c^2 +3c+2}{2}.
\end{align*}
\label{prop-a2}
\end{prop}

\begin{proof}
Since $L$ has $c$ crossings, the number of regions of $L$ is $c+2$. 
Hence, we have the upper bound $\binom{c+2}{2}$, the number of edges of $K_{c+2}$ since $\overline{G} \subseteq K_{c+2}$. 
For the lower bound, let $x_1, x_2, \dots , x_{c+2}$ be the regions of $L$. 
Suppose $x_i$ is a $k_i$-gon. 
Then $x_i$ has at most $2 k_i$ regions around $x_i$. 
For each region $x_i$, the number of regions which is disconnected to $x_i$ is at least $(c+2)-2k_i -1$, where $c+2$ implies the number of all regions of $L$. 
In total, we have 
\begin{align*}
a_2 \geq \sum_{i=1}^{c+2} \frac{c+1-2k_i}{2} = \frac{(c+1)(c+2)}{2}- \sum_{i=1}^{c+2}k_i,
\end{align*}
where we divided the summation by 2 to fix the double counting. 
Using Adams, Shinjo and Tanaka's formula in \cite{AST} 
\begin{align*}
c= \frac{2C_2 +3C_3 + 4C_4 + \dots}{4},
\end{align*}
we have 
\begin{align*}
 \sum_{i=1}^{c+2} k_i =2C_2 +3C_3 + 4C_4 + \dots =4c
\end{align*}
for connected irreducible link projections, where $C_j$ denotes the number of $j$-gons. 
Hence, we have 
\begin{align*}
a_2 \geq \frac{(c+1)(c+2)}{2}-4c = \frac{c^2 -5c +2}{2}.
\end{align*}
\end{proof}

\section{I-generating functions of $(2,n)$-torus link projections}
\label{section-pf2}

In this section, we investigate the I-generating functions for the standard projections of $(2,n)$-torus links, and prove the following theorem. 

\vspace{3mm}
\begin{thm}
The I-generating function $f(L)$ of a standard projection  $L$ of the $(2,n)$-torus link $( n \geq 4 )$ is
\begin{align*}
f(L) =1+(n+2)x+ \sum _{k=2}^{\lfloor \frac{n}{2} \rfloor} \binom{n-k}{k} \frac{n}{n-k} x^k .
\end{align*}
\label{formula-T2n}
\end{thm}

\noindent For a standard projection $T(2,n)$ of a $(2,n)$-torus link, we give labels of regions $x_1, x_2, \dots , x_{n+2}$ as shown in Figure \ref{fig-2n}.

\begin{figure}[h]
\centering
\includegraphics[width=40mm]{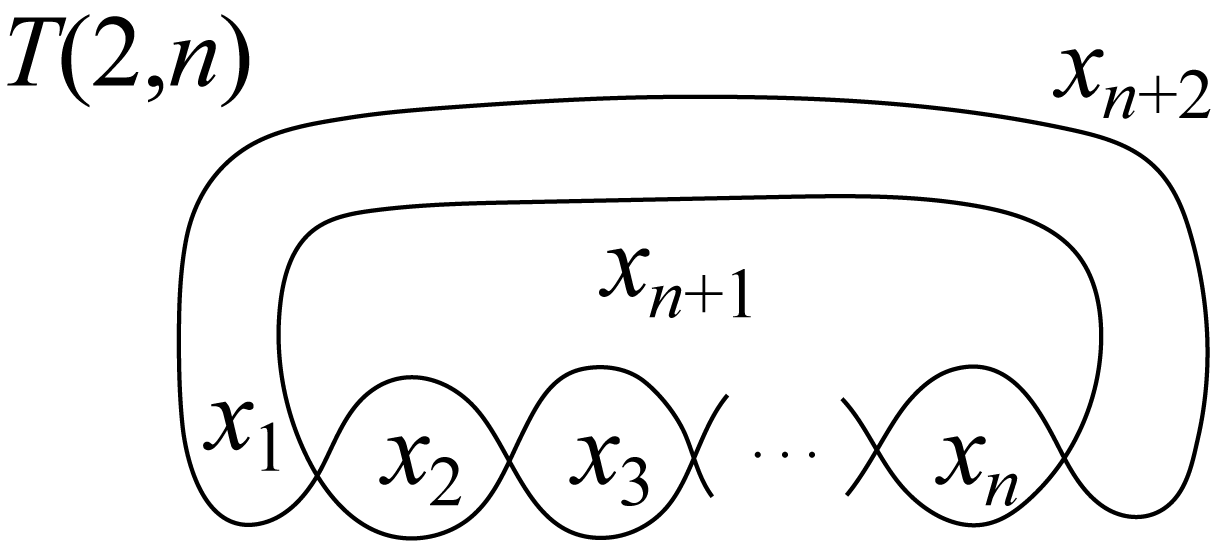}
\caption{A standard projection $T(2,n)$ of a $(2,n)$-torus link with regions $x_1, x_2, \dots , x_{n+2}$.}
\label{fig-2n}
\end{figure}

\noindent We notice that the regions $x_{n+1}$ and $x_{n+2}$ are connected with all the other regions. 
Also, when $n \geq 3$, each region $x_i$ for $1 \leq i \leq n$ is connected to the four regions, $x_{n+1}$, $x_{n+2}$, and the neighboring two bigons. 
In the region-connect graph $G$, therefore, the vertices $x_{n+1}$ and $x_{n+2}$ have degree $(n+1)$, and the others have degree four. 
As for $\overline{G}$, the vertices $x_{n+1}$ and $x_{n+2}$ have degree zero, and the others have degree $(n-3)$. 
Moreover, in $\overline{G}$, the vertices $x_1, x_2, \dots , x_n$ forms a complete graph $K_n$ missing the following $n$ edges, $x_1 x_2, x_2 x_3, x_3 x_4, \dots , x_n x_1$. 
For example, when $n=5$, we have $G$ and $\overline{G}$ as shown in Figure \ref{fig-25}, and the polynomials $h=(1+x_1)(1+x_2) \dots (1+x_7)$, $g= 1+ x_1 +x_2 + \dots + x_7 + x_1 x_3 + x_3 x_5 +x_5 x_2 + x_2 x_4 + x_4 x_1$, and $f=1+7x+5x^2$. 

\begin{figure}[h]
\centering
\includegraphics[width=90mm]{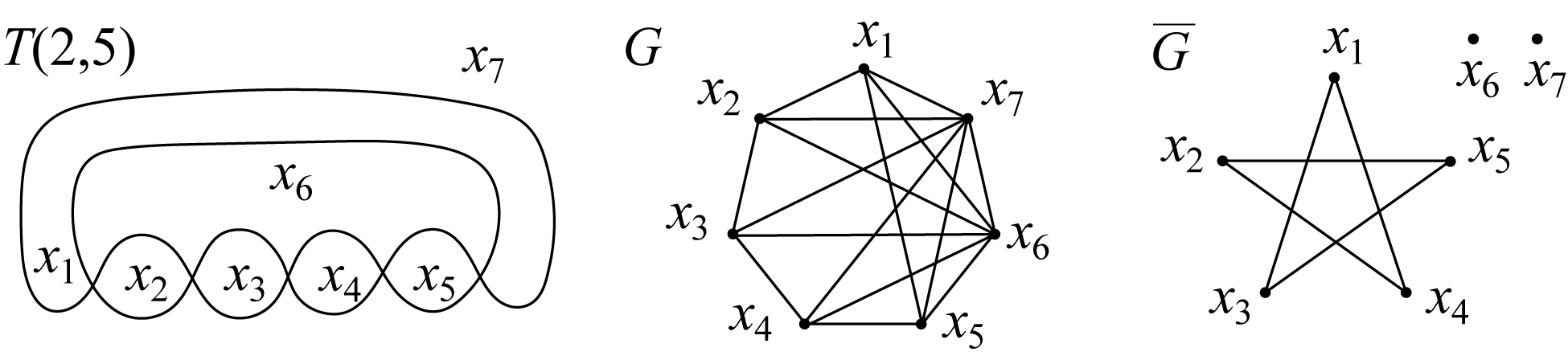}
\caption{$T(2,5)$ and its $G$ and $\overline{G}$.}
\label{fig-25}
\end{figure}

\noindent To prove Theorem \ref{formula-T2n}, we prepare a (perhaps well-known) formula. 

\vspace{3mm}
\begin{lem}
Suppose there are $m$ white balls in a line. 
Denote by $\left( \binom{m}{n} \right)$ the number of cases to paint $n$ of $m$ balls red so that the painted balls are not next to each other. Then, 
\begin{align*}
\left( \binom{m}{n} \right) = \binom{m-n+1}{n}.
\end{align*}
\end{lem}

\vspace{1mm}

\begin{proof}
At first, suppose $(m-n)$ white balls in a line. 
Then there are $(m-n+1)$ rooms between balls or on a side. 
The number of ways to put the painted $n$ red balls there is $\binom{m-n+1}{n}$.
\end{proof}

\vspace{1mm}

\noindent Now we prove Theorem \ref{formula-T2n}. \\

\noindent {\it Proof of Theorem \ref{formula-T2n}}. 
Let $\overline{G}$ be the region-disconnected graph of $T(2,n)$ with the $(n+2)$ regions labeled as in Figure \ref{fig-2n}. 
Recall that the vertices $x_{n+1}$ and $x_{n+2}$ have degree zero. 
For the subgraph of $\overline{G}$ consisting of the vertices $x_1, x_2, \dots , x_n$ and all the edges, we discuss how many complete graphs $K_k$ are included. 

Let $f=a_0 +a_1 x + a_2 x^2 + \dots + a_r x^r$ be the I-generating function. 
For the maximum degree, we have $r= \text{maxdeg}f= \lfloor \frac{n}{2} \rfloor$ because we can take $\lfloor \frac{n}{2} \rfloor$ non-neighboring vertices and cannot for $\lfloor \frac{n}{2} \rfloor +1$ from $\overline{G}$, where we say $x_i$ and $x_j$ are neighboring when $|i-j|=1$ or $|i-j|= n-1$. 

For coefficients, we have $a_0 = 1$ as the empty set of vertices. 
We have $a_1 = n+2$, the number of vertices (or $K_1$) of $\overline{G}$, as mentioned in Proposition \ref{prop-func}. 
We have $a_2 = (n-3) \frac{n}{2}$, the number of edges (or $K_2$) of $\overline{G}$. 
More precisely, fix a vertex $x_i$. 
Then $x_i$ has $n-3$ choices of vertices to form an edge in $\overline{G}$. 
Recall that $x_i$ cannot be connected with the neighboring vertices. 
Multiple $n-3$ by $n$ because we can consider the same thing for the $n$ vertices. 
Divide it by two since each edge is double counted. 
Hence, we have $a_2 = (n-3) \frac{n}{2}$. 

Next, for $k \geq 3$, we obtain $a_k = ( \binom{n-3}{k-1}) \frac{n}{k}$ in the following way. 
Fix a vertex $x_i$. 
Then $x_i$ has $n-3$ non-neighboring edges and has $( \binom{n-3}{k-1} )$ choices to take other $k-1$ vertices to form $K_k$ so that any pair of vertices is not neighboring. 
Multiple $( \binom{n-3}{k-1} )$ by $n$, and divide it by $k$ since each $K_k$ is counted $k$ times. 
Then, 

\begin{align*}
a_k &= \left( \binom{n-3}{k-1} \right) \frac{n}{k} = \binom{(n-3)-(k-1)+1}{k-1} \frac{n}{k} = \binom{n-k-1}{k-1} \frac{n}{k} \\
&= \frac{n-k}{n-k} \frac{(n-k-1)(n-k-2) \dots (n-2k+1)}{(k-1)!} \frac{n}{k} = \binom{n-k}{k} \frac{n}{n-k}.
\end{align*}

\noindent This holds when $k=2$, too. \hfill$\square$ \\

\vspace{1mm}

\noindent The I-generating functions $f_n$ of $T(2,n)$ are listed below for $n=1$ to $12$. 

\begin{align*}
f_1 &= 1 + 3x \\
f_2 &= 1 + 4x \\
f_3 &= 1 + 5x \\
f_4 &= 1 + 6x + 2x^2 \\
f_5 &= 1 + 7x + 5x^2 \\
f_6 &= 1 + 8x + 9x^2 + 2x^3 \\
f_7 &= 1 + 9x + 14x^2 + 7x^3 \\
f_8 &= 1 + 10x + 20x^2 + 16x^3 + 2x^4 \\
f_9 &= 1 + 11x + 27x^2 + 30x^3 + 9x^4 \\
f_{10} &= 1 + 12x + 35x^2 + 50x^3 + 25x^4 + 2x^5 \\
f_{11} &= 1 + 13x + 44x^2 + 77x^3 + 55x^4 + 11x^5 \\
f_{12} &= 1 + 14x + 54x^2 + 112x^3 + 105x^4 +36x^5 +2x^6
\end{align*}

\vspace{1mm}

\noindent We have the following formula. 

\vspace{1mm}

\begin{prop}
For the I-generating functions $f_n= f_n (x)$ of $T(2,n)$, we have 
\begin{align*}
f_n - f_{n-1} - x f_{n-2} = -2x^2
\end{align*}
for $n \geq 4$. 
\label{prop-recurrence}
\end{prop}

\vspace{1mm}

\begin{proof}
Let $b_k$ be the coefficient of $x^k$ of $f_n - f_{n-1} - x f_{n-2}$. 
When $k \geq 3$, the $k$th coefficient of $f_n$ is $\binom{n-k}{k} \frac{n}{n-k}$, that of $f_{n-1}$ is $\binom{n-1-k}{k} \frac{n-1}{n-1-k}$, and the $(k-1)$th coefficient of $f_{n-2}$ is $\binom{(n-2)-(k-1)}{k-1} \frac{n-2}{(n-2)-(k-1)} = \binom{n-k-1}{k-1} \frac{n-2}{n-k-1}$. 
Hence, 
\begin{align*}
b_k = \binom{n-k}{k} \frac{n}{n-k} - \binom{n-k-1}{k} \frac{n-1}{n-k-1} - \binom{n-k-1}{k-1} \frac{n-2}{n-k-1}.
\end{align*}

\noindent Using the formula $\binom{n+1}{r}= \binom{n}{r-1} + \binom{n}{r}$, we have $\binom{n-k}{k} = \binom{n-k-1}{k-1} + \binom{n-k-1}{k}$ and 

\begin{align*}
b_k &= \binom{n-k-1}{k-1} \left( \frac{n}{n-k} - \frac{n-2}{n-k-1} \right) + \binom{n-k-1}{k} \left( \frac{n}{n-k} - \frac{n-1}{n-k-1} \right) \\
&= \frac{(n-k-1)!}{(k-1)!(n-2k)!} \frac{n-2k}{(n-k)(n-k-1)} - \frac{(n-k-1)!}{k! (n-2k-1)!} \frac{k}{(n-k)(n-k-1)} \\
&= \frac{(n-k-1)!}{(k-1)!(n-2k-1)!} \frac{1}{(n-k)(n-k-1)}- \frac{(n-k-1)!}{(k-1)!(n-2k-1)!} \frac{1}{(n-k)(n-k-1)} \\
&= 0.
\end{align*}

\noindent For $k \leq 2$, $b_2 = \binom{n-2}{2} \frac{n}{n-2} - \binom{n-3}{2} \frac{n-1}{n-3} -n =-2$, $b_1 = (n+2)-(n+1)-1=0$ and $b_0 = 1-1=0$. 
Hence, we have $b_2 = -2$ and $b_k =0$ for $k \neq 2$, and $f_n - f_{n-1} - xf_{n-2} = -2 x^2$ holds for $n \geq 4$. 
\end{proof}

\vspace{1mm}

\noindent From Proposition \ref{prop-recurrence}, we have a recursive relationship about the number of isolated-region sets. 

\begin{cor}
For the I-generating function $f_n (x)$ of $T(2,n)$ of $n \geq 4$, we have 
\begin{align*}
f_n (1)=f_{n-1}(1)+ f_{n-2}(1)-2.
\end{align*}
\end{cor}

\section*{Acknowledgment}

The authors thank Professor Rama Mishra for her support and encouragement.
The second author's work was partially supported by JSPS KAKENHI Grant Number JP21K03263.


\begin{thebibliography}{99} 

\bibitem{AST} C. C. Adams, R. Shinjo and K. Tanaka, Complementary regions of knot and link diagrams, Ann. Comb. {\bf 15} (2011), 549--563. 

\bibitem{ahara-suzuki} K. Ahara and M. Suzuki, An integral region choice problem on knot projection, J. Knot Theory Ramifications {\bf 21} (2012), 1250119.

\bibitem{fenn} M. H. Fenn, R. Rim\'anyi and C. Rourke, The braid-permutation group, Topology {\bf 39} (1997), 123--135. 

\bibitem{lecture} A. Kawauchi, Lectures on knot theory (in Japanese), Kyoritsu Shuppan (2007).

\bibitem{quan} A. Kawauchi and A. Shimizu, Quantization of the crossing number of a knot diagram, Kyungpook Math. J. {\bf 55} (2015), 741--752.

\bibitem{unknotting-w} Z. Li, F. Lei and J. Wu, On the unknotting number of a welded knot, J. Knot Theory Ramifications {\bf 26} (2017), 1750004. 

\bibitem{non-tri} T. Mahato, R. Mishra and S. Joshi, Non-triviality of welded knots and ribbon torus-knots, preprint (arXiv:2404.00436).

\bibitem{warping-1} A. Ohya and A. Shimizu, Lower bounds for the warping degree of a knot projection, J. Knot Theory Ramifications {\bf 31} (2022), 2250091.

\bibitem{rourke} C. Rourke, What is a welded link?, Intelligence of low dimensional topology 2006, 263--270, Ser. Knots Everything {\bf 40} (2007), World Sci. Publ., Hackensack, NJ.

\bibitem{s-virtual} S. Satoh, Virtual knot representation of ribbon torus-knots, J. Knot Theory Ramifications {\bf 9} (2000), 531--542.

\bibitem{warping-2} A. Shimizu, Prime alternating knots of minimal warping degree two, J. Knot Theory Ramifications {\bf 29} (2020), 2050060.

\bibitem{warping-3} A. Shimizu, The warping degree of a knot diagram, J. Knot Theory Ramifications {\bf 19} (2010), 849--857. 

\bibitem{warping-p} A. Shimizu, The warping polynomial of a knot diagram, J. Knot Theory Ramifications {\bf 21} (2012), 1250124.



\end{thebibliography}
\end{document}